\documentclass[12pt,american]{article}
\usepackage[T1]{fontenc}
\usepackage[latin9]{inputenc}
\usepackage{babel}
\usepackage{amsmath}
\usepackage{amsthm}
\usepackage{amssymb}
\usepackage[unicode=true,
 bookmarks=true,bookmarksnumbered=true,bookmarksopen=false,
 breaklinks=true,pdfborder={0 0 1},backref=false,colorlinks=true]
 {hyperref}
\hypersetup{pdftitle={Compound Poisson Processes: Potentials, Green Measures and Random Times},
 pdfauthor={Jos{\'e} L. da Silva and Yuri G. Kondratiev},
 pdfsubject={Compound Poisson Processes: Potentials, Green Measures and Random Times},
 pdfkeywords={Markov processes, compound Poisson process, Green measures, time change, asymptotic behavior},
 linkcolor=black,citecolor=black,urlcolor=black,filecolor=black}

\makeatletter
\theoremstyle{plain}
\newtheorem{thm}{\protect\theoremname}[section]
\theoremstyle{plain}
\newtheorem{cor}[thm]{Corollary}
\theoremstyle{plain}
\newtheorem{lem}[thm]{\protect\lemmaname}
\theoremstyle{plain}

\theoremstyle{plain}

\theoremstyle{definition}
\newtheorem{defn}[thm]{\protect\definitionname}
\theoremstyle{definition}
\newtheorem{example}[thm]{\protect\examplename}
\theoremstyle{remark}
\newtheorem{rem}[thm]{\protect\remarkname}

\theoremstyle{plain}
\usepackage{bbm}

\usepackage[deletedmarkup=xout]{changes}
\numberwithin{equation}{section}

\newcommand{\R}{{\mathbb R}}
\newcommand{\X}{{\R^d}}

\renewcommand{\G}{\mathcal G}
\newcommand{\B}{{\mathcal B}}

\newcommand{\E}{{\mathbb E}}

\newcommand{\la}{\lambda}

\makeatother

\providecommand{\definitionname}{Definition}
\providecommand{\examplename}{Example}
\providecommand{\lemmaname}{Lemma}
\providecommand{\remarkname}{Remark}
\providecommand{\theoremname}{Theorem}

\begin{document}
\title{Compound Poisson Processes: Potentials, Green Measures and Random
Times}
\author{\textbf{Yuri Kondratiev}\\
BiBoS, University of Bielefeld,\\
D-33615 Bielefeld, Germany, and\\
 Interdisciplinary Center for Complex Systems (IDCCS), \\
 National Dragomanov University, 01004 Kiev, Ukraine\\
 Email: kondrat@math.uni-bielefeld.de\and\textbf{Jos{\'e} Lu{\'i}s
da Silva},\\
Faculty of Exact Sciences and Engineering,\\
CIMA, University of Madeira, Campus da Penteada,\\
 9020-105 Funchal, Portugal.\\
 Email: joses@staff.uma.pt}
 
\date{\today}
\maketitle
\begin{abstract}
In this paper we study the existence of Green measures for Markov
processes with a nonlocal jump generator. The jump generator has no second moment and satisfies a suitable condition on its Fourier transform. We also study the same
problem for certain classes of random time changes Markov processes
with jump generator.


\end{abstract}
\tableofcontents{}

\section{Introduction}

Compound Poisson processes are well known in  stochastic analysis as an important class of stochastic processes with
independent increments \cite{Skorohod1991}. In mathematical physics these processes correspond to the random walks in
the continuum with a given jump kernels. Note that random walks in physics 
are considered as diffusions in the wide sense contrary to the much more restricted notion
in  stochastics. That's why random walks are very useful in modeling
 several phenomena observed in concrete physical systems as, e.g., 
amorphous media, random impurity models, etc.

The central problem to be analyzed is how the properties of jump kernels will be
 reflected in the behavior of the random walk, precisely their time asymptotic. 
 One of the many possibilities to analyze the long time behavior of general Markov processes is to study
 their potentials.  The general theory of potentials is well elaborated, see, e.g., \cite{Blumenthal1968}.
 But an application of this theory to particular examples of Markov processes is not a trivial work.
 
 In the present paper we are concerned with the case of compound Poisson processes. 
 We have the following questions to address. 
\begin{enumerate}
\item  Under which conditions on the jump kernel do we have  the existence of the corresponding  potential?
\item How to obtain the integral representations for the potentials by means of measures?
\item Consider the random time changes in the compound Poisson processes. 
How to define the potentials and Green measures for these processes? 
\end{enumerate}
 
 Let us describe each of these questions in more details.  
 Denote by $X$ the compound Poisson process starting from $x\in\X$. The process $X$ is associated to a generator $L$ defined by a kernel $a:\X\longrightarrow \R$.  Given a function $f:\X\longrightarrow\R$ from a proper class, the potential of $f$ is defined by
 \[
 V(x,f):=\int_{0}^{\infty}\E^x[f(X(t))]\,\mathrm{d}t.
 \]
The class of functions $f$ for which the potential $V(x,f)$ exists turns out to be the Banach space $CL(\X):=C_b(\X)\cap L^1(\X)$ with its natural norm $\| f\|_{CL}:=\| f\|_\infty+\| f\|_1$. The usual conditions on the kernel $a$ are positive, continuous, bounded, symmetric and finite second moment. In this paper rather the finite second moment assumption is replaced by a rather weaker condition (see condition (A) in page \pageref{condition-A}) formulated in terms of its Fourier transform. More precisely, there is $A>0$ and $0<\alpha\le2$ such that
\begin{equation}\label{Fourier-a}
    \hat{a}(k)=1-A|k|^{\alpha}+\mathrm{o}(|k|^{\alpha})\quad\mathrm{as}\quad k\to0.
\end{equation}
For example, the Cauchy distribution has no finite second moment but its Fourier transform admits and expansion  as in \eqref{Fourier-a} with $A=\alpha=1$, see Example~\ref{exa:Gauss-Cauchy} below.

The measures which appear in question 2 above for the representations of the potentials we call Green measures. The name is due to the fact that these measures are fundamental solutions to the Markov generators of the associated compound Poisson process. They are Green functions as generalized functions in the classical sense.
These generalized function are realized as measures in our case.  In addition, we would like to mention that our results are strongly dependent on the dimension of the location spaces. Actually, the Green measures 
 may be defined only for the dimensions $d\geq \alpha$, see Theorem~\ref{thm:average} below. For that reason, the study of the asymptotic behaviors of additive functionals of classical Markov processes in dimensions $d<\alpha$ needs careful renormalizations.
 
 We also introduce the notion of random Green measures which  is a more detailed characteristic of the stochastic process $X$. These measures rise from the representation of the corresponding random potentials of $f$. That this, given $f\in CL(\X)$ we define the random variable   
 \begin{equation*}
     Y^x(f):=\int_0^\infty f(X(t))\,\mathrm{d}t
 \end{equation*}
and would like to obtain their integral representations as
\begin{equation*}
    Y^x(f)(\omega)=\int_\X f(y)\G(x,\mathrm{d}y,\omega),\quad \forall \omega\in\Omega.
\end{equation*}
 It is clear that the Green measures for compound Poisson processes are the mathematical expectations of random Green measures, that is, \begin{equation*}
     \G(x,\mathrm{d}y)=\E^x[\G(x,\mathrm{d}y,\cdot)].
 \end{equation*} 
 
 Finally, on the third questions we are interested in the random time changes on the compound Poisson processes, that is, processes defined as $Y(t):=X(D(t))$, $t\ge0$, where $D$ is a random time change. As random times changes we have chosen a class of inverse of subordinators $S$ satisfying assumption (H), see page \pageref{condition-H} for details.  Clearly, a random time
 change will destroy the Markov property of the stochastic process $X$. Nevertheless, we try to define the associated potentials and Green measures. It turns out that the definition of Green measure for this class of stochastic processes need to be renormalized. More precisely, the definition of the Green measures as introduced in Section~\ref{sec:MP-GM} for this class of stochastic processes does not exist for any dimension $d$. Thus, we define a family of proper normalized Green measures such that the limit reduces to the Green measures of the initial stochastic processes $X$, see Theorem~\ref{thm:main-RTC-GM}. As a consequence, we obtain an asymptotic result for the Green measures associated to the random time changes of the compound Poisson process, see Theorem~\ref{thm:asymptotic-FKE}.
 
The paper is organized as follows. In Section \ref{sec:MP-GM} we
give the general definition of Green measure of a Markov process in terms of its transition probabilities, 
see Definition~\ref{def:MP-GM}. We also notice that Green measures are the fundamental solutions to the Markov generators, see \eqref{eq:FS}.
In Section~\ref{sec:Jump-Generators} we introduce the jump generator of a compound Poisson process given by a kernel $a$. We state the assumptions on the kernel $a$ under which the Green measure for the associated process starting from $x\in\X$ exists. We emphasize assumption (A) on the Fourier transform of $a$, see page \pageref{condition-A}. In addition, the class of functions $f$ for which the potential $V(x,f)$ exists is described. In Section~\ref{sec:RGM} we  investigate the integral representation of the random potential $Y^x(f)(\omega)$. The resulting measure is called random Green measure. In Theorem~\ref{thm:representation1} we show the existence of the random Green measure and identify the class of functions $f$ which allows the integral representation.
Finally, in Section~\ref{sec:MP-RT} we study the Green measure of random time changes of the compound Poisson process. 
The class of subordinators determining the time changes is identified and the Green measures are obtained by a limit process, see Theorem~\ref{thm:asymptotic-FKE}.

\section{Markov Processes and Green Measures }

\label{sec:MP-GM}Let $(\Omega,\mathcal{F},\mathbb{P})$ be a probability
space and $X$ a time-homogeneous Markov processes on $\X$
starting from $x\in\X$. That is, $X:[0,\infty)\times\Omega\longrightarrow\X$
with $X(0,w)=x$, for every $w\in\Omega$, see for example \cite{Blumenthal1968,Dynkin1965,Meyer1967,Revuz-Yor-94}
for more details. The Markov process $X$ may be defined in a standard
way by its transition probabilities $P_{t}(x,B)$. The functions $P_{t}(x,B)$ give the
probability of the transition from the point $x\in\X$
at time $t=0$ to the Borel set $B\in\mathcal{B}(\X)$
at time $t>0$. In certain cases, the transition probabilities $P_{t}(x,\cdot)$
may have densities $p_{t}(x,\cdot)$ with respect to the Lebesgue
measure in $\X$, that is, for any Borel set $B\in\mathcal{B}(\X)$
we have
\[
P_{t}(x,B)=\int_{B}p_{t}(x,y)\,\mathrm{d}y.
\]
Equivalently, we have $P_{t}(x,\mathrm{d}y)=p_{t}(x,y)\,\mathrm{d}y$
and the density $p_{t}(x,\cdot)$ is also known as heat kernel.

Let $f:\X\to\R$ be a Borel function and consider the, so-called potential
of $f$, defined by 
\begin{equation}
V(x,f):=\int_{0}^{\infty}\mathbb{E}^{x}[f(X(t)]\,\mathrm{d}t=\int_{0}^{\infty}\int_{\X}f(y)\,P_{t}(x,\mathrm{d}y)\,\mathrm{d}t,\label{eq:potential-f}
\end{equation}
In probability theory, the notion of potential is well known, see
e.g., \cite{Blumenthal1968,Revuz-Yor-94}. Nevertheless, to establish
the existence of $V(x,f)$ is a very difficult problem and the class
of admissible $f$ have to be analyzed for each process $X$ separately.
Assume that for a certain class of functions $f$ (e.g., $f\in C_{0}(\X)$
the space of continuous functions with compact support) the potential
$V(x,f)$ is well defined, then $V(x,f)$ may be represented by a
Radon measure on $\X$ as
\[
V(x,f)=\int_{\X}f(y)\,\G(x,\mathrm{d}y).
\]
The measure $\G(x,\cdot)$ we called the Green measure for
the process $X$. It is possible to express the measure $\G(x,\cdot)$
using the transience property of the process $X$. In fact, it follows
from \eqref{eq:potential-f}, applying Fubini's theorem, that
\[
\G(x,\mathrm{d}y)=\int_{0}^{\infty}P_{t}(x,\mathrm{d}y)\,\mathrm{d}t
\]
assuming the existence of this object as a Radon measure on $\X$.
That is, for any bounded Borel set $B\in\mathcal{B}_{b}(\X)$
we have 
\[
\G(x,B)=\int_{0}^{\infty}P_{t}(x,B)\,\mathrm{d}t<\infty.
\]
In addition, if it happen that $P_{t}(x,\mathrm{d}y)=p_{t}(x,y)\,\mathrm{d}y$,
then the function 
\[
g(x,y):=\int_{0}^{\infty}p_{t}(x,y)\,\mathrm{d}t,\quad\forall y\in\X,
\]
is called the Green function of the Markov process $X$. In this case
the Green measure is given by $\G(x,\mathrm{d}y)=g(x,y)\,\mathrm{d}y$.
The existence of the Green function $g(x,y)$ for a given process
$X$ or transition probability $P_{t}(x,\cdot)$, even for simple
classes of Markov processes, is not always guaranteed. As an example
we consider $X=B$ the 1-dimensional Brownian motion (Bm for short),
then for every $f\in L^{1}(\R)$ and any $t>0$ we have 
\[
\int_{0}^{t}f\big(B(s)\big)\,\mathrm{d}s=\int_{\R}f(y)L_{t}(y)\,\mathrm{d}y,
\]
where $L_{t}(y)$ is the local time of Bm up to time $t$ at the point
$y$. For the definition and properties of local time of 1-dimensional
Bm, see, e.g., \cite{Takacs1995,Takacs1995a,Borodin1984,Borodin2002}.
Now, the asymptotic behavior as $t\to\infty$ of $\int_{0}^{t}f(B(s))\,\mathrm{d}s$,
as the 1-dimensional Bm $B$ is recurrent, we have $L_{\infty}(y)=\infty$
for all $y\in\R$, consequently, the integral functional $\int_{0}^{\infty}f(B(t))\,\mathrm{d}t$,
roughly speaking, does not exists. 

Nevertheless, Green functions for certain classes of Markov processes
are well known in probability theory, see, e.g., \cite{CGL20,Grigoryan2018a}
and references therein for more details.

Thus, we give the definition of Green measure for a Markov process
$X.$
\begin{defn}[Green measure]\label{def:MP-GM}
Let $X(t)$, $t\ge0$, be a Markov process starting from $x\in\X$
with transition probability $P_{t}(x,\cdot)$. The Green measure of
$X$ is defined by 
\[
\G(x,B):=\int_{0}^{\infty}P_{t}(x,B)\,\mathrm{d}t,\;\;B\in\B_{b}(\X),
\]
or 
\[
\int_{\X}f(y)\G(x,\mathrm{d}y)=\int_{\X}f(y)\int_{0}^{\infty}P_{t}(x,dy)\,\mathrm{d}t,\;\;f\in C_{0}(\X)
\]
whenever these integrals exist. 
\end{defn}

In the following, we give an equivalent definition of the Green measure
$\G(x,\cdot)$ of $X$ as the fundamental solution of the
corresponding generator $L$. We may start with a Markov semigroup
$T(t),t\geq0$, that is, a family of linear operators on a Banach
space $E$. As a Banach space $E$, we may use bounded measurable
functions $B(\X)$, bounded continuous functions $C_{b}(\X)$ or the
Lebesgue spaces $L^{p}(\X)$, $p\ge1$, depending on each particular
case. The semigroup $T(t)$, $t\ge0$, is associated with a Markov
process $\{X(t),t\geq0\;|\;\mathbb{P}^{x},x\in\X\}$ if for any $t\ge0$
we have

\[
(T(t)f)(x)=\mathbb{E}^{x}[f(X(t))]=\int_{\X}f(y)P_{t}(x,\mathrm{d}y),\quad f\in E.
\]
The transition probabilities $P_{t}(x,\cdot)$ are constructed from
the semigroup by choosing $f=\mathbbm{1}_{B}$, $B\in\B(\X)$, that
is, 
\[
P_{t}(x,B)=(T(t)\mathbbm{1}_{B})(x).
\]

We recall the definition of the Markov generator $L$ of the semigroup
$T(t)$, $t\ge0$.
\begin{defn}[Markov generator]
We define the generator $(L,D(L))$ of the Markov semigroup $T(t)$,
$t\ge0$, by 
\[
D(L):=\left\{ f\in E\,\middle|\,\frac{T(t)f-f}{t}\;\mathrm{converges\;in}\;E\;\mathrm{when}\;t\to0^{+}\right\} 
\]
and for every $f\in D(L)$, $y\in\X$ 
\[
(Lf)(y):=\lim_{t\to0^{+}}\frac{(T(t)f)(y)-f(y)}{t}.
\]
Then $D(L)$ (the domain of $L$) is a linear subspace of $E$ and
$L:D(L)\longrightarrow E$ is a linear operator. There are several
known properties of the generator $L$ and we have the full description
of the Markov generators via the so-called maximum principle, see,
e.g., \cite{Grigoryan2014}. 
\end{defn}

From the relation between semigroup and generator we have
\begin{equation}
\begin{split}\int_{0}^{\infty}(T(t)f)(x)\,\mathrm{d}t & =\int_{0}^{\infty}\mathbb{E}^{x}[f(X(t))]\,\mathrm{d}t=\int_{0}^{\infty}\int_{\X}f(y)P_{t}(x,\mathrm{d}y)\,\mathrm{d}t\\
 & =\int_{\X}f(y)\G(x,\mathrm{d}y)=-(L^{-1}f)(x).
\end{split}
\label{eq:FS}
\end{equation}
 for every $f\in C_{0}(\X)$. Because any Radon measure defines a
generalized function, then we may write 
\[
\G(x,\mathrm{d}y)=g(x,y)\,\mathrm{d}y,
\]
where $g(x,\cdot)\in D'(\X)$ is a positive generalized function for
all $x\in\X$. In view of (\ref{eq:FS}) the Green measure is the
fundamental solution corresponding to the operator $L$. Note that
the existence and regularity of this fundamental solution produces
a description of admissible Markov processes for which the Green measure
exists. 

\section{Jump Generators and Green Measures}

\label{sec:Jump-Generators}Let $a:\X\to\R$ be a fixed kernel function
which is symmetric, positive, bounded, continuous, and integrable
with $\int_{\X}a(x)\,\mathrm{d}x=1$. Using the kernel
$a$ we define an operator $L=L_{a}$ on $E$ (as stated above) by
\begin{equation}
(Lf)(x):=\int_{\X}a(x-y)[f(y)-f(x)]\,\mathrm{d}y=(a*f)(x)-f(x),\quad x\in\X.\label{eq:jump-generator}
\end{equation}
The operator $L$ is self-adjoint in $L^{2}(\X)$ and
a bounded linear operator in all $L^{p}(\X)$, $p\ge1$.
We call this operator the jump generator with jump kernel $a$. The
associated Markov process $X$ is of a pure jump type and is known
in stochastic as compound Poisson process (or continuous time random walk),
see \cite{Skorohod1991}
for more details. If we denote $u(t,x):=\mathbb{E}^{x}[f(X(t))]$,
$x\in\X$, $t>0$, and $f\in E$, then $u(t,x)$ satisfies
the Kolmogorov equation 
\begin{equation}
\begin{cases}
\partial_{t}u(t,x) & =Lu(t,x),\\
u(0,x) & =f(x).
\end{cases}
\label{eq:KEquation}
\end{equation}
If $u(t,x)$ denotes the density of a population at the position $x$
at time $t$, and $a(x-y)$ is thought of as the probability distribution
of jumping from location $y$ to location $x$, then $(a*u)(t,x)$
gives the rate at which individuals arrive at the position $x$ from
other places.

Many analytic properties of the jump generator $L$ were investigated
in \cite{Grigoryan2018,Kondratiev2017,Kondratiev2018}. Here we recall
some of these properties necessary in what follows.

First notice that, due to the symmetry of the kernel $a$, its Fourier
image is given by

\[
\hat{a}(k)=\int_{\X}\mathrm{e}^{-i(k,y)}a(y)\,\mathrm{d}y=\int_{\X}\cos((k,y))a(y)\,\mathrm{d}y.
\]
In addition, it is easy to see that 
\[
\hat{a}(0)=1,\quad|\hat{a}(k)|\le1,\;k\neq0,\quad\hat{a}(k)\to0,\;k\to\infty.
\]
On the other hand, since $L$ is a convolution operator, the Fourier
image of $L$ is the multiplication operator by 
\[
\hat{L}(k)=\hat{a}(k)-1
\]
which is the symbol of $L$.

In what follows we make the following additional assumption on the
jump kernel $a$.
\begin{description}\label{condition-A}
\item [{(A)}] The jump kernel $a$  has an expansion of the form 
\[
\hat{a}(k)=1-A|k|^{\alpha}+\mathrm{o}(|k|^{\alpha})\quad\mathrm{as}\quad k\to0,
\]
with $A>0$ and $0<\alpha\le2$.
\end{description}
\begin{example}\label{exa:Gauss-Cauchy}
A well known example of a jump kernel $a$ which satisfies Assumption
(A) is the Gaussian kernel: 
\[
a(x)=\frac{1}{2^{d/2}}\mathrm{e}^{-|x|^{2}/4}\;\Longrightarrow\;\hat{a}(k)=\mathrm{e}^{-|k|^{2}}=1-|k|^{2}+\mathrm{o}(|k|^{2}),\quad k\to0.
\]
The Gaussian kernel has second moment, but in general the second moment
of $a$ may fail, so that more general expansions may occur: $\hat{a}(k)=1-A|k|^{\alpha}+\mathrm{o}(|k|^{\alpha})$
with $\alpha\in(0,2)$. As an example in one dimension we mention
the Cauchy distribution 
\[
a(x)=\frac{1}{1+x^{2}}\;\Longrightarrow\;\hat{a}(k)=1-|k|+\mathrm{o}(|k|),\quad k\to0.
\]
\end{example}

For $\la\in(0,\infty)$, let $R_{\la}(L)$ denote the resolvent of
the operator $L$, that is, $R_{\la}(L)=(\lambda-L)^{-1}$. Using
the Neumann series for the operator $(\lambda-L)^{-1}$, we obtain
\begin{equation}
(\lambda I-L)^{-1}=\frac{1}{1-\lambda}I+\frac{1}{1+\lambda}A_{\lambda},\label{eq:resolvent}
\end{equation}
where $I$ is the identity operator and $A_{\lambda}$ is the operator
\[
A_{\lambda}=(1+\lambda)(\lambda I-L)^{-1}-I=\sum_{n=1}^{\infty}\frac{a^{*n}}{(1+\lambda)^{n}},
\]
where $a^{*n}=a*a*\dots *a$ is the $n$th times convolution of $a$ with itself. If $\G_{\la}(x,y)$, $x,y\in\X$, $\la\in(0,\infty)$ denotes the
resolvent kernel of $R_{\la}(L)$, then it follows from \eqref{eq:resolvent}
that the kernel $\G_{\la}(x,y)$ admits the decomposition (cf.~Eq.
(4) in \cite{Kondratiev2018})
\begin{equation}\label{kernel-resolvent}
    \G_{\la}(x,y)=\frac{1}{1+\la}\big(\delta(x-y)+G_{\la}(x-y)\big),
\end{equation}
where $G_{\lambda}$ is the kernel of the operator $A_{\lambda}$. The kernel $G_\lambda(x)$ is given by 
\begin{equation}
G_{\la}(x)=\sum_{n=1}^{\infty}\frac{a_{n}(x)}{(1+\la)^{n}},\label{eq:Green-kernel}
\end{equation}
and $a_{n}(x)=a^{\ast n}(x)$ is the $n$-fold convolution of the
jump kernel $a$. In addition, the kernel $G_{\lambda}(x)$ is given
in terms of its inverse Fourier transform by
\[
G_{\lambda}(x)=\frac{1}{(2\pi)^{d}}\int_{\X}\mathrm{e}^{\mathrm{i}(k,x)}\frac{\hat{a}(k)}{1+\lambda-\hat{a}(k)}\,\mathrm{d}k.
\]
Notice that the resolvent kernel $\G_{\la}(x,y)$ has a singular part,
$\delta(x-y)$ and a regular part $G_{\la}(x-y)$. The Green function,
as a generalized function, has the form 
\[
\G_{0}(x)=\delta(x)+G_{0}(x).
\]
In particular, the Fourier representation for $G_{0}$ has the form 

\begin{equation}
G_{0}(x)=\frac{1}{(2\pi)^{d}}\int_{\X}\mathrm{e}^{i(k,x)}\frac{\hat{a}(k)}{1-\hat{a}(k)}\,\mathrm{d}k.\label{eq:Fourier-Green-kernel}
\end{equation}

\begin{thm}
\label{thm:average}Let $a$ be a jump kernel as defined above which
satisfies assumption \emph{(A)}. Then the Green measure for the random walk
exists if $d>\alpha$.
\end{thm}
\begin{proof}
It is sufficient to show that the regular part of the resolvent kernel
$G_{0}(x)$, given by \eqref{eq:Fourier-Green-kernel}, is finite
for all $x\in\X$. It follows from the assumptions on
$a$ that the integral in \eqref{eq:Fourier-Green-kernel} exists
for all $x\in\X$ due to the integrability $(1-\hat{a}(k))^{-1}$
at $k=0$ by the assumption (A).
\end{proof}

\begin{rem}
\begin{enumerate}
    \item The existence of the Green measure shown in Theorem~\ref{thm:average}  is based on the Fourier representation of the regular part (i.e., $G_\lambda$) of the resolvent kernel $\G_\lambda(x,y)$ in \eqref{kernel-resolvent} at $\lambda=0$.
	\item On the other hand, we are interested on the description of the  class of integrable functions $f:\X\longrightarrow\R$ for which the potential of $f$ in \eqref{eq:potential-f} exits.
\end{enumerate}
\end{rem}
In order to identify this class of functions $f$ we define the following Banach space:
\begin{equation}\label{eq:Banach-space-CL}
CL(\R^{d}):=C_{b}(\X)\cap L^{1}(\X)    
\end{equation}
with the norm
\[
\|f\|_{CL}:=\|f\|_{\infty}+\|f\|_{1}:=\sup_{x\in\X}|f(x)|+\int_{\X}|f(x)|\,\mathrm{d}x.
\]
\begin{cor}
For any $f\in CL(\X)$ the potential $V(x,f)$ of $f$ is well defined and we have
\[
V(x,f)=\int_0^\infty\E^x[f(X(t)]\,\mathrm{d}t=\int_{\X}f(y)\G(x,\mathrm{d}y).
\]
\end{cor}
\begin{proof}It follows from \eqref{kernel-resolvent} with $\lambda=0$ that 
\[
V(x,f)=f(x)+\int_\X f(y)G_0(x-y)\,\mathrm{d}y,\quad \forall x\in\R.
\]
From Theorem~\ref{thm:average} and \eqref{eq:Fourier-Green-kernel} the function $G_0(x)$ is uniformly bounded, hence por any $f\in CL(\X)$ the potential $V(x,f)$ is well defined and claim follows.
\end{proof}

\section{Potentials and Random Green Measures}
\label{sec:RGM}
In this section we are interested in another object associated with
the compound Poisson process $X$
introduced in Section~\ref{sec:Jump-Generators}. Namely, having
in mind that $X$ starts from $x\in\X$, given a function $f:\X\longrightarrow\R$,
we define the random variables $Y^{x}(f)$ by 
\begin{equation}
Y^{x}(f):=\int_{0}^{\infty}f(X(t))\,\mathrm{d}t.\label{eq:random-potential}
\end{equation}
We call $Y^{x}(f)$ the random potentials of $f$. The relation between
$Y^{x}(f)$ and the Green measure $\G(x,\mathrm{d}y)$ introduced
in the previous sections is clear: 
\[
\E^{x}[Y^{x}(f)]=\int_{\X}f(y)\G(x,\mathrm{d}y),\quad \forall f\in CL(\X).
\]
Our aim is to show that the random variables $Y^{x}(f)$ admit
the representation 
\[
Y^{x}(f)(\omega)=\int_{\X}f(y)\G(x,\mathrm{d}y,\omega),\quad \omega\in \Omega.
\]
with a vector valued random measure $\G(x,\mathrm{d}y,\omega)$, called random Green measure from now on. These measures take values
in $L^{1}(\mathbb{P}):=L^{1}(\Omega,\mathbb{P})$ and they are finite
in all bounded Borel sets $B\in\mathcal{B}_{b}(\X)$.

The following theorem is the main result of this section on the existence
of the random Green measure for the random potential \eqref{eq:random-potential}.
\begin{thm}
\label{thm:representation1} For each $x\in\X$ the operator 
\[
Y^{x}:CL(\X)\longrightarrow L^{1}(P)
\]
has a unique representation given, for all $f\in CL(\X)$ and every
$\omega\in\Omega$, by 
\begin{equation}
Y^{x}(f)(\omega)=\int_{\X}f(y)\,\G(x,\mathrm{d}y,\omega)\label{repr}
\end{equation}
with a vector valued $\sigma$-additive $($in the strong topology
of $L^{1}(\mathbb{P}))$ Radon measure $\G(x,\mathrm{d}y,\omega)$
on $\B_{b}(\X)$.
\end{thm}

\begin{proof}
The proof follows the same arguments as Theorem 2.1 in \cite{Kondratiev2022} to which we address the interested reader.
\end{proof}
\begin{rem}
For particular models, it is possible to show that $\mathbb{E}[\G(x,\X,\cdot)]<\infty$,
that is $\G(x,\X,\omega)<\infty$ for $\mathbb{P}$-a.a.~$\omega\in\Omega$.
In addition, the random variables $Y^x(f)$ are not degenerated, that is, its
variance is positive. An example is provided in Proposition 2.3 in
\cite{Kondratiev2022}.
\end{rem}

\section{Green Measures for Time Changed Markov Processes }

\label{sec:MP-RT}Time change Markov processes have found many applications
in probability theory, see \cite{Magdziarz2015} for a detailed discussion
and several related references. In this section we investigate the
existence of the Green measure $\G(x,\mathrm{d}y)$ for the
time changes of the compound Poisson process $X$ introduced before.
As time change we choose a class $D(t)$, $t\ge0$, of inverse subordinators
to be specified later, see Subsection~\ref{subsec:The-Class-TC}.
Hence, we are interested in a new process $Y=Z(t)$, $t\ge0$, which
is constructed by a random time change $D(t)$, $t\ge0$, in $X$.
That is, we define the stochastic process  $Z$ by 
\begin{equation}
Z(t):=X(D(t)),\quad t\ge0.\label{eq:TC-MP}
\end{equation}
The connection between the process $X$ and $Y$ were investigated
first in the paper \cite{Mura_Taqqu_Mainardi_08} and later by other
authors, see e.g., \cite{Toaldo2015} and references therein. The
situation is described as follows. Define the function $u(t,x)$
by 
\[
u(t,x):=\E^{x}[f(X(t))],\;t>0,\;x\in\X,\;f\in E.
\]
The function $u(t,x)$ is the solution of the Kolmogorov equation
\eqref{eq:KEquation} with $L$ the jump generator defined in \eqref{eq:jump-generator}.
We define a similar function for $Z(t)$: 
\[
v(t,x)=\E^{x}[f(Z(t))].
\]
Then this function satisfies the following fractional Kolmogorov equation:
\begin{equation}
\mathbb{D}_{t}^{(k)}v(t,x)=Lv(t,x),\label{FKE}
\end{equation}
where $L$ is the jump generator of $X$, $k\in L^1_{\mathrm{loc}
}(\R_+)$ is a function defined
in terms of the L{\'e}vy measure of the subordinator and $\mathbb{D}_{t}^{(k)}$ is the differential-convolution
operator defined in \eqref{eq:general-derivative}. Moreover, if
the time change $D(t)$ admits a density $\rho_{t}$, $t>0$, then
the subordination formula holds:

\begin{equation}
v(t,x)=\int_{0}^{\infty}u(\tau,x)\rho_{t}(\tau)\,\mathrm{d}\tau.\label{SUB}
\end{equation}
In terms of the marginal distribution $\mu_{t}^{x}$ of $X(t)$ and $\nu_{t}^{x}$
of $Z(t)$ the subordination relations for these distributions is
given by 
\begin{equation}
\nu_{t}^{x}=\int_{0}^{\infty}\mu_{\tau}^{x}\rho_{t}(\tau)\,\mathrm{d}\tau.\label{SUBM}
\end{equation}
Below we use these relations to study the renormalized Green measure
associated to the stochastic process $Z$.

\subsection{The Class of Times Changes}

\label{subsec:The-Class-TC}Let $S=\{S(t),\;t\ge0\}$ be a subordinator
without drift starting from zero with Laplace exponent $\Phi$ and
L{\'e}vy measure $\sigma$, see \cite{Bertoin96} for more details.
The Laplace transform of $S(t)$, $t\ge0$, is given by
\[
\mathbb{E}(\mathrm{e}^{-\lambda S(t)})=\mathrm{e}^{-t\Phi(\lambda)},\quad\forall\lambda\ge0,
\]
and $\Phi$ admits the representation 
\begin{equation}
\Phi(\lambda)=\int_{(0,\infty)}(1-\mathrm{e}^{-\lambda\tau})\,\mathrm{d}\sigma(\tau).\label{eq:Levy-Khintchine}
\end{equation}
The L{\'e}vy measure $\sigma$ is supported in $[0,\infty)$ and satisfies
\begin{equation*}
    \int_0^\infty(1\wedge\tau)\,\mathrm{d}\sigma(\tau)<\infty.
\end{equation*}
We define the kernel $k$ by 
\begin{equation}
k:(0,\infty)\longrightarrow(0,\infty),\;t\mapsto k(t):=\sigma\big((t,\infty)\big)\label{eq:k}
\end{equation}
and denote its Laplace transform by $\mathcal{K}$, that is, 
\begin{equation}
\mathcal{K}(\lambda):=\int_{0}^{\infty}\mathrm{e}^{-\lambda t}k(t)\,\mathrm{d}t,\quad\forall\lambda\ge0.\label{eq:LT-k}
\end{equation}
The following relation holds $\Phi(\lambda)=\lambda\mathcal{K}(\lambda)$,
$\forall\lambda\ge0.$  

In what follows we make the following assumption on the Laplace exponent
$\Phi$ and $\mathcal{K}$ associated to the subordinator $S$.
\begin{description}\label{condition-H}
\item [(H)] $\Phi$ is a complete Bernstein function, that is, $\sigma$
has a completely monotone density with respect to the Lebesgue measure, and the functions $\Phi$ and $\mathcal{K}$
satisfy 
\begin{equation}
\mathcal{K}(\lambda)\to\infty,\text{ as \ensuremath{\lambda\to0}};\quad\mathcal{K}(\lambda)\to0,\text{ as \ensuremath{\lambda\to\infty}};\label{eq:H1}
\end{equation}
\begin{equation}
\Phi(\lambda)\to0,\text{ as \ensuremath{\lambda\to0}};\quad\Phi(\lambda)\to\infty,\text{ as \ensuremath{\lambda\to\infty}}.\label{eq:H2}
\end{equation}
\end{description}
Classical examples of subordinators satisfying assumption (H) are the
$\alpha$-stable process and the Gamma process. The concrete expressions
of $k,\mathcal{K}$, and $\Phi$ for these subordinators are shown
in e.g., Example 2.1 in \cite{KdS20}. 

Denote by $D$ the inverse process of the subordinator $S$, that
is, 
\begin{equation}
D(t):=\inf\{s\ge0\mid S(s)\ge t\}\label{eq:inverse-sub},\quad t\ge0,
\end{equation}
and its marginal density by $\rho_{t}$, that is, for any $\tau\ge0$,
\[
\rho_{t}(\tau)\,\mathrm{d}\tau=\partial_{\tau}\mathbb{P}(D(t)\le\tau)=\partial_{\tau}\mathbb{P}(S(\tau)\ge t)=-\partial_{\tau}\mathbb{P}(S(\tau)<t).
\]

The marginal density $\rho_{t}$ plays an important role below. In
the following lemma we collect the most important properties of $\rho_{t}$.
For the proof, see \cite{Kochubei11} or Lemma~3.1 in \cite{Toaldo2015}. 
\begin{lem}
\label{lem:t-LT-G}The $t$-Laplace transform of the density $\rho_{t}(\tau)$
is given by 
\begin{equation}
\int_{0}^{\infty}\mathrm{e}^{-\lambda t}\rho_{t}(\tau)\,\mathrm{d}t=\mathcal{K}(\lambda)\mathrm{e}^{-\tau\lambda\mathcal{K}(\lambda)}.\label{eq:LT-G-t}
\end{equation}
The double ($\tau,t$)-Laplace transform of $\rho_{t}(\tau)$ is 
\begin{equation}
\int_{0}^{\infty}\int_{0}^{\infty}\mathrm{e}^{-p\tau}\mathrm{e}^{-\lambda t}\rho_{t}(\tau)\,\mathrm{d}t\,\mathrm{d}\tau=\frac{\mathcal{K}(\lambda)}{\lambda\mathcal{K}(\lambda)+p}.\label{eq:double-Laplace}
\end{equation}
\end{lem}

Given the subordinator $S$ and the kernel $k\in L_{\mathrm{loc}}^{1}(\mathbb{R}_{+})$
as defined in \eqref{eq:k}, we may introduce the differential-convolution operator of a continuous function $f$, such that $k*f$ is almost everywhere differentiable, by 
\begin{equation}
\big(\mathbb{D}_{t}^{(k)}f\big)(t)=\frac{\mathrm{d}}{\mathrm{d}t}\int_{0}^{t}k(t-\tau)f(\tau)\,\mathrm{d}\tau-k(t)f(0),\;t>0.\label{eq:general-derivative}
\end{equation}
The operator $\mathbb{D}_{t}^{(k)}$ is also known as generalized
fractional derivative, see \cite{Kochubei11} and references therein.

Now we define the class of admissible $k(t)$ and state the key result
(see Theorem~\ref{thm:main-result} below) to establish the existence
of the Green measures for the random time changes of the stochastic process $X$ 
defined in \eqref{eq:TC-MP}, see Theorem~\ref{thm:main-RTC-GM}.

In what follows, we use the notation $f\sim g$ as $t\to\infty$ and
say that $f$ and $g$ are asymptotically equivalent when $t\to\infty$
if $\lim_{t\to\infty}\frac{f(t)}{g(t)}=1$.

\begin{defn}[Admissible kernels - $\mathbb{K}(\mathbb{R}_{+})$]
The subset $\mathbb{K}(\mathbb{R}_{+})\subset L_{\mathrm{loc}}^{1}(\mathbb{R}_{+})$
of admissible kernels $k$ is defined by those elements in $L_{\mathrm{loc}}^{1}(\mathbb{R}_{+})$
satisfying (A) such that for certain $s_{0}>0$ with
\begin{equation}
\liminf_{\lambda\to0+}\frac{1}{\mathcal{K}(\lambda)}\int_{0}^{s_0/\lambda}k(t)\,dt>0\label{eq:A1}
\end{equation}
and 
\begin{equation}
\lim_{\genfrac{}{}{0pt}{2}{\frac{t}{r}\to1}{t,r\to\infty}}\left(\int_{0}^{t}k(s)\,ds\right)\left(\int_{0}^{r}k(s)\,ds\right)^{-1}=1.\label{eq:A2}
\end{equation}
\end{defn}

The asymptotic relation between the marginal density $\rho_{t}(\tau)$
of $D(t)$ and the kernel $k\in\mathbb{K}(\mathbb{R}_{+})$ is given
in the following theorem. For the proof, see Theorem~5.3 in \cite{KKdS19}. 

\begin{thm}
\label{thm:main-result}Let $\tau\in[0,\infty)$ be fixed and $k\in\mathbb{K}(\mathbb{R}_{+})$
a given admissible kernel. Define the map $\rho_{\cdot}(\tau):[0,\infty)\longrightarrow\mathbb{R}_{+}$,
$t\mapsto\rho_{t}(\tau)$ such that $\int_{0}^{\infty}\mathrm{e}^{-\lambda t}\rho_{t}(\tau)\,dt$
exists for all $\lambda>0$. Then 
\begin{equation}
\lim_{t\to\infty}\left(\int_{0}^{t}\rho_{s}(\tau)\,ds\right)\left(\int_{0}^{t}k(s)\,ds\right)^{-1}=1\label{eq:limit-G-k}
\end{equation}
or 
\[
M_{t}\big(\rho_{t}(\tau)\big):=\frac{1}{t}\int_{0}^{t}\rho_{s}(\tau)\,ds\sim\frac{1}{t}\int_{0}^{t}k(s)\,ds=:M_{t}\big(k(t)\big),\quad t\to\infty
\]
and $M_{t}\big(\rho_{t}(\tau)\big)$ is uniformly bounded in $\tau\in\mathbb{R}_{+}$. 
\end{thm}

\subsection{Renormalized Green Measures}

We are now ready to study the Green measure of the
random time changes of the compound Poisson process $X$ by the inverse subordinator
$D(t)$. That is, the process $Z(t)=X(D(t))$, $t\ge0$, introduced
at the beginning of Section~\ref{sec:MP-RT}. Below we distinguish
between the Green measure of the Markov process $X$ from that of
$Z$. Thus, we use $\G^{X}(x,\mathrm{d}y)$ for the Green
measure of $X$ while $\G^{Z}(x,\mathrm{d}y)$ denotes the
Green measure associated to $Z$. 

Following the general approach from Section~\ref{sec:MP-GM}, we
would like to show that the following integral 
\[
\G^{Z}(x,\mathrm{d}y):=\int_{0}^{\infty}\nu_{t}^{x}(\mathrm{d}y)\,\mathrm{d}t=\int_{0}^{\infty}\int_{0}^{\infty}\mu_{\tau}^{x}(\mathrm{d}y)\rho_{t}(\tau)\,\mathrm{d}\tau\,\mathrm{d}t
\]
exists and obtain the representation 
\begin{equation}
\int_{0}^{\infty}v(t,x)\,\mathrm{d}t=\int_{0}^{\infty}\mathbb{E}^{x}[f(Z(t))]\,\mathrm{d}t=\int_{\X}f(y)\G^{Z}(x,\mathrm{d}y).\label{eq:representation-TC-MP}
\end{equation}
Formulated in this way, the Green measure of $Z$ does not exists
for the class of Markov processes $X$ (in any dimension $d$) and
subordinators introduced above, see Lemma 5 in \cite{KdS20}. Thus,
the representation \eqref{eq:representation-TC-MP} will not be possible.
To overcome this problem, we shall consider, instead, the renormalized
Green measure. More precisely, we would like to find the following
limit 
\begin{equation}
\G_{r}^{Z}(x,\mathrm{d}y):=\lim_{T\to\infty}\frac{1}{N(T)}\int_{0}^{T}\nu_{t}^{x}(\mathrm{d}y)\,\mathrm{d}t.\label{eq:Normalized-GM}
\end{equation}
The normalization $N(T)$ in \eqref{eq:Normalized-GM} is chosen in
a suitable way so that the renormalized Green measure $\G_{r}^{Z}(x,\mathrm{d}y)$
coincides with the Green measure $\G^{X}(x,\mathrm{d}y)$
of the initial Markov process $X$. More precisely, we have the following
theorem borrowed from \cite{KdS20} and included here for completeness.
\begin{thm}
\label{thm:main-RTC-GM}Assume that the Markov process $X$ in
$\X$, $d\geq\alpha$, has a Green measure $\G^{X}(x,\mathrm{d}y)$
and define 
\begin{equation}
N(T):=\int_{0}^{T}k(s)\,\mathrm{d}s,\quad T\ge0.\label{eq:normalization}
\end{equation}
Then the renormalized Green measure for $Z$ exists and we have 
\[
\G_{r}^{Z}(x,\mathrm{d}y)=\G^{X}(x,\mathrm{d}y).
\]
\end{thm}

\begin{proof}
Using the subordination relation \eqref{SUBM} the renormalized Green
measure $\G_{r}^Z(x,\mathrm{d}y)$ may be written as 
\[
\G_{r}^{Z}(x,\mathrm{d}y)=\lim_{T\to\infty}\frac{1}{N(T)}\int_{0}^{T}\int_{0}^{\infty}\mu_{\tau}^{x}(\mathrm{d}y)\rho_{t}(\tau)\,\mathrm{d}\tau\,\mathrm{d}t.
\]
Now using Fubini's theorem and Theorem \ref{thm:main-result} it follows
that 
\[
\G_{r}^{Z}(x,\mathrm{d}y):=\lim_{T\to\infty}\frac{1}{N(T)}\int_{0}^{T}\nu_{t}^{x}(\mathrm{d}y)\,\mathrm{d}t=\int_{0}^{\infty}\mu_{t}^{x}(\mathrm{d}y)\,\mathrm{d}t=\G^{X}(x,\mathrm{d}y).
\]
This shows the claim and completes the proof. 
\end{proof}
We are now ready to state the main result of this section.
\begin{thm}
\label{thm:asymptotic-FKE}Under the assumptions of Theorem \ref{thm:average}
holds
\[
\frac{1}{N(t)}\int_{0}^{t}v(s,x)\,\mathrm{d}s\sim\int_{\X}f(y)\G^X(x,\mathrm{d}y),\quad t\to\infty.
\]
\end{thm}

\begin{proof}
Using the subordination formula (\ref{SUB}) we obtain 
\[
\frac{1}{N(t)}\int_{0}^{t}v(s,x)\,\mathrm{d}s=\frac{1}{N(t)}\int_{0}^{t}\int_{0}^{\infty}u(\tau,x)\rho_{s}(\tau)\,\mathrm{d}\tau\,\mathrm{d}s.
\]
Again it follows from Fubini theorem, Theorem \ref{thm:main-result}
and the definition of $N(t)$ that 
\[
\frac{1}{N(t)}\int_{0}^{t}v(s,x)\,\mathrm{d}s=\int_{0}^{\infty}u(\tau,x)\left(\frac{1}{N(t)}\int_{0}^{t}\rho_{s}(\tau)\,\mathrm{d}s\right)\,\mathrm{d}\tau\underset{t\to\infty}{\longrightarrow}\int_{0}^{\infty}u(\tau,x)\,\mathrm{d}\tau.
\]
Then the result of the theorem follows from Theorem \ref{thm:average}. 
\end{proof}

\subsubsection*{Acknowledgments}

This work has been partially supported by Center for Research in Mathematics
and Applications (CIMA) related with the Statistics, Stochastic Processes
and Applications (SSPA) group, through the grant UIDB/MAT/04674/2020
of FCT-Funda{\c c\~a}o para a Ci{\^e}ncia e a Tecnologia, Portugal.

The financial support from the Ministry for Science and Education of
Ukraine through Project 0119U002583 is gratefully acknowledged.


\begin{thebibliography}{KMPZ18}

\bibitem[Ber96]{Bertoin96}
J.~Bertoin.
\newblock {\em L\'evy processes}, volume 121 of {\em Cambridge Tracts in
  Mathematics}.
\newblock Cambridge University Press, Cambridge, UK, 1996.

\bibitem[BG68]{Blumenthal1968}
R.~M. Blumenthal and R.~K. Getoor.
\newblock {\em {Markov Processes and Potential Theory}}.
\newblock Academic Press, 1968.

\bibitem[Bor84]{Borodin1984}
A.~N. Borodin.
\newblock {Distribution of integral functionals of a Brownian motion process}.
\newblock {\em J. Math. Sci.}, 27(5):3005--3022, 1984.

\bibitem[BS02]{Borodin2002}
A.N. Borodin and P.~Salminen.
\newblock {\em {Handbook of Brownian Motion - Facts and Formulae}}.
\newblock Birkh{\"a}user Basel, 2002.

\bibitem[CGL21]{CGL20}
J.~Cao, A.~Grigor'yan, and L.~Liu.
\newblock {Hardy's inequality and Green function on metric measure spaces}.
\newblock {\em J.\ Funct.\ Anal.}, 281:109020, 2021.

\bibitem[Dyn65]{Dynkin1965}
E.~B. Dynkin.
\newblock {\em Markov Processes}.
\newblock Springer, 1965.

\bibitem[GH14]{Grigoryan2014}
A.~Grigor'yan and J.~Hu.
\newblock Heat kernels and {G}reen functions on metric measure spaces.
\newblock {\em Canad. J. Math.}, 66(3):641--699, 2014.

\bibitem[GHH18]{Grigoryan2018a}
A.~Grigor'yan, E.~Hu, and J.~Hu.
\newblock {Two-sided estimates of heat kernels of jump type Dirichlet forms}.
\newblock {\em Adv.\ Math.}, 330:433--515, 2018.

\bibitem[GKPZ18]{Grigoryan2018}
A.~Grigor'yan, Yu.~G. Kondratiev, A.~Piatnitski, and E.~Zhizhina.
\newblock Pointwise estimates for heat kernels of convolution-type operators.
\newblock {\em Proc. Lond. Math. Soc. (3)}, 114(4):849--880, 2018.

\bibitem[KdS21]{KdS20}
Yu.~G. Kondratiev and J.~L. da~Silva.
\newblock {Green Measures for Time Changed Markov Processes}.
\newblock {\em Methods Funct.\ Anal.\ Topology}, 27(3):227--236, 2021.

\bibitem[KdS22]{Kondratiev2022}
Y.~G. Kondratiev and J.~L. da~Silva.
\newblock {Random potentials for Markov processes}.
\newblock {\em Applicable Analysis}, 2022.
\newblock In press.

\bibitem[KKdS20]{KKdS19}
A.~Kochubei, Yu.~G. Kondratiev, and J.~L. da~Silva.
\newblock Random time change and related evolution equations. {T}ime asymptotic
  behavior.
\newblock {\em Stochastics and Dynamics}, 4:2050034--1--24, 2020.

\bibitem[KMPZ18]{Kondratiev2018}
Yu.~G. Kondratiev, S.~Molchanov, A.~Piatnitski, and E.~Zhizhina.
\newblock Resolvent bounds for jump generators.
\newblock {\em Appl.\ Anal.}, 97(3):323--336, 2018.

\bibitem[KMV17]{Kondratiev2017}
Yu.~G. Kondratiev, S.~Molchanov, and B.~Vainberg.
\newblock Spectral analysis of non-local {S}chr{\"o}dinger operators.
\newblock {\em J. Funct. Anal.}, 273(3):1020--1048, 2017.

\bibitem[Koc11]{Kochubei11}
A.~N. Kochubei.
\newblock General fractional calculus, evolution equations, and renewal
  processes.
\newblock {\em Integral Equations Operator Theory}, 71(4):583--600, 2011.

\bibitem[Mey67]{Meyer1967}
P.-A. Meyer.
\newblock {\em {Processus de Markov}}, volume~26 of {\em Lecture Notes in
  Mathamatics}.
\newblock Springer, 1967.

\bibitem[MS15]{Magdziarz2015}
M.~Magdziarz and R.~L. Schilling.
\newblock {Asymptotic properties of Brownian motion delayed by inverse
  subordinators}.
\newblock {\em Proceedings of the American Mathematical Society},
  143(10):4485--4501, 2015.

\bibitem[MTM08]{Mura_Taqqu_Mainardi_08}
A.~Mura, M.~S. Taqqu, and F.~Mainardi.
\newblock Non-{M}arkovian diffusion equations and processes: analysis and
  simulations.
\newblock {\em Phys.\ A}, 387(21):5033--5064, 2008.

\bibitem[RY99]{Revuz-Yor-94}
D.~Revuz and M.~Yor.
\newblock {\em Continuous martingales and {B}rownian motion}, volume 293 of
  {\em Grundlehren der Mathematischen Wissenschaften [Fundamental Principles of
  Mathematical Sciences]}.
\newblock Springer-Verlag, Berlin, 3rd edition, 1999.

\bibitem[Sko91]{Skorohod1991}
A.~V. Skorohod.
\newblock {\em {Random Processes with Independent Increments}}, volume~47 of
  {\em Mathematics and its applications (Soviet series)}.
\newblock Springer, 1991.

\bibitem[Tak95a]{Takacs1995a}
L.~Tak{\'a}cs.
\newblock {B}rownian local times.
\newblock {\em J. Appl. Math., Stoch. Anal.}, 8:209--232, 1995.

\bibitem[Tak95b]{Takacs1995}
L.~Tak{\'a}cs.
\newblock On the local time of the {B}rownian motion.
\newblock {\em Ann. Appl. Probab.}, 5:741--756, 1995.

\bibitem[Toa15]{Toaldo2015}
B.~Toaldo.
\newblock Convolution-type derivatives, hitting-times of subordinators and
  time-changed {$C_0$}-semigroups.
\newblock {\em Potential Anal.}, 42(1):115--140, 2015.

\end{thebibliography}
\end{document}